\documentclass[11pt]{amsart}
\usepackage{verbatim, latexsym, amssymb, amsmath }
\def\R{\mathbb R}
\def\N{\mathbb N}

\def\H{\mathbb H}

\def\h2{\mathrm{area}}

\def\a {\left\lvert\Sigma \right\rvert}
\def\b_i{\left\lvert\Sigma_i\right\rvert}
\def\d{\mathrm{div}}
\newcommand{\fint}{\mathop{\int\makebox(-15,2){\rule[4pt]{.9em}{0.6pt}}\kern-4pt}\nolimits}
\def\tf {\mathring{A}}
\def\n{\lvert\mathring{A}\rvert}

\def\dr{\lvert\partial_r^{\top}\rvert}
\def\trr{\partial_r^{\top}}
\def\ds{\lvert\partial_s^{\top}\rvert}
\def\ts{\partial_s^{\top}}
\def\dw_i{\lvert\widehat{\nabla} w_t\rvert}

\def\tr{\mathrm{tr}_{g_0}}

\newtheorem{thm}{Theorem}[section]
\newtheorem{lemm}[thm]{Lemma}

\newtheorem{prop}[thm]{Proposition}
\theoremstyle{remark}

\theoremstyle{definition}
\newtheorem{defi}[thm]{Definition}

\title{Existence and Uniqueness of constant mean curvature foliation of asymptotically hyperbolic $3$-manifolds II}

\author{{Andr\'e Neves} ${}^{\dagger}$}
\email{aneves@math.princeton.edu}
\address{Fine Hall, Princeton University,
Princeton, NJ 08544, USA}

\author{Gang Tian}
\email{tian@math.princeton.edu}
\address{Fine Hall, Princeton University,
Princeton, NJ 08544, USA}

\thanks{\quad\ ${}^{\dagger}$\ The author was partially supported by NSF grant DMS-06-04164.}
\begin{document}

\maketitle \markboth{Existence and uniqueness of constant mean curvature foliations} { Andr\'e Neves and Gang
Tian} \maketitle

\begin{abstract}
In a previous paper, the authors showed that metrics which are asymptotic to  Anti-de Sitter-Schwarzschild metrics with positive mass admit a unique foliation by stable spheres with constant mean curvature. In this paper we extend that result to all asymptotically  hyperbolic metrics for which the trace of the mass term is positive. We do this by combining the Kazdan-Warner obstructions with a theorem due to De Lellis and M\"uller.
\end{abstract}

\section{Introduction}

In a previous paper \cite{neves}, the authors showed that metrics which are asymptotic to  Anti-de Sitter-Schwarzschild metrics with positive mass admit a unique foliation by stable spheres with constant mean curvature. This metrics are a special case of asymptotically hyperbolic metrics and they arise naturally as initial conditions for the Einstein equations with a negative cosmological constant.  This result was motivated by an earlier work of Huisken and Yau \cite{Huisken3} where they studied a similar question for asymptotically flat metrics, more precisely, metrics which are asymptotic to  a Schwarzschild with positive mass.

Because being asymptotic to  Anti-de Sitter-Schwarzschild metric is a very restrictive condition, the purpose of this paper is to extend the results of \cite{neves} to a more general class of asymptotically hyperbolic metrics. Some estimates of \cite{neves} rely on the fact that the ambient metric is very close to being  Anti-de Sitter-Schwarzschild and so some new arguments are needed. 

Fortunately, a similar question was considered by Jan Metzger in \cite{metzger} where he extended the results of Huisken and Yau not only to admit foliations with prescribed mean curvature but also to admit  metrics  which are small perturbations of  Schwarzschild metrics where the perturbation term has the same order has the mass term. We note that some of the estimates done by Huisken and Yau do not apply to the setting considered by Metzger and so some new ideas were needed. Most notably, he makes a very nice observation regarding Simon's identity \cite[Identity (2.3)]{metzger} and  then uses quite effectively \cite[Section 4]{metzger} a theorem due to De Lellis and M\"uller \cite{DeLellis}. This allows him to derive optimal apriori estimates  for stable constant mean curvature spheres. Finally, he uses a continuity method to deform a coordinate sphere in Schwarzschild space into a stable constant mean curvature sphere for the metric he is considering, where the apriori estimates he derived assure him that the spheres cannot drift, i..e, their centers do not rush off to infinity.

Inspired by \cite{metzger}, we will use a very sharp estimate due to De Lellis and M\"uller \cite{DeLellis} to show that  stable spheres in asymptotically hyperbolic $3$-manifolds are very close to coordinate spheres for some coordinate system. The main part of the argument will be to show that these coordinate systems have to approximate a {\em predetermined} coordinate system, i.e., no drifting occurs. This will be accomplished using the Kazdan-Warner obstructions \cite{Kazdan}. We note that it is in this step that the positivity of the mass must be used because such phenomena does not occur in hyperbolic space. We then apply a continuity argument to prove existence and uniqueness.

We end up  the introduction with the following remark. The work done here simplifies somewhat the previous work of the authors but does not fully generalizes it for two reasons. Firstly, in this paper, the foliations considered have the property that the difference between the outer radius and the inner radius is uniformly bounded, while in \cite{neves} we allow for foliations more general than that. Secondly,  the continuity method uses the fact that stable constant mean curvature spheres in Anti-de Sitter-Schwarzschild are unique. This was proven in \cite{neves} and requires more refined estimates that the ones we use in this paper.

In Section \ref{sec2}, we give some definitions and state the main result. In Section \ref{prelim}, Section \ref{intestimate}, and Section \ref{Intrin}, we adapt the work done in \cite{neves} to our new setting. Most of the proofs will be obvious modifications of the work done in \cite{neves}. In Section \ref{ap} we use De Lellis and M\"uller theorem in order to show that stable constant mean curvature spheres are very close to some coordinate spheres. In Section \ref{Unique}, we use the Kazdan-Warner obstructions to show that the stable constant mean curvature spheres have to approximate coordinate spheres for a {\em fixed} coordinate system. Finally, in Section \ref{U} we use a continuity method to prove existence and uniqueness of foliations by stable spheres with constant mean curvature.

\section{Definitions and statement of main theorem}\label{sec2}

\subsection{Definitions} 

Given a complete noncompact Riemannian $3$-manifold $(M,g)$, we denote its connection by $D$, the Ricci curvature by $
Rc$, and the scalar curvature by $R$.  The induced connection on a surface $\Sigma\subset M$ is denoted by $\nabla$, the exterior unit normal by $\nu$ (whenever its defined), the mean curvature by $H$, and the surface area by  $|\Sigma|$.

In what follows $g_0$ denotes the standard metric on $S^2$.

\begin{defi}\label{defi}
A complete noncompact Riemannian $3$-manifold $(M,g)$ is said to be asymptotically hyperbolic if the following are true:
\begin{itemize}
\item[(i)] There is a compact set $K\subset\subset M$  such that $M\setminus K$ is diffeomorphic to $\R^3\setminus B_{r_1}(0)$.

\item [(ii)] With respect to the spherical coordinates induced by the above diffeomorphism, the metric can be written as
$$g=dr^2+\sinh^2r \,g_0+h/(3\sinh r) + Q$$
where $h$ is a symmetric $2$-tensor on $S^2$ and 
$$\lvert Q\rvert+\lvert D Q\rvert +\lvert D^2 Q\rvert+\lvert D^3 Q\rvert\leq C_1\exp(-4r)$$
for some constant $C_1$.
\end{itemize}
\end{defi}
   The above definition is stated differently from the one given in \cite{wang}  (see also \cite{herzlich}). Nonetheless, using a simple substitution of variable $$t=\ln\left(\frac{\sinh(r/2)}{\cosh(r/2)}\right),$$ they can be  seen to be equivalent.
 
 Note that a given coordinate system on $M\setminus K$ induces a radial function $r(x)$ on $M\setminus K$. With respect to this coordinate system, we define the {\em inner radius} and {\em outer radius} of a surface $\Sigma\subset M\setminus K$ to be
 $$\underline{r}=\sup\{r\,|\, B_r(0)\subset \Sigma\}\quad\mbox{and}\quad\overline{r}=\inf\{r\,|\, \Sigma\subset B_r(0)\}$$
respectively. Furthermore, we denote the coordinate spheres induced by a coordinate system by
$$\{|x|=r\}:=\{x\in M\setminus K\,\,|\,\,r(x)=r\}$$
and the radial vector by $\partial_r$. Moreover, $\partial_r^{\top}$ stands for the tangential projection of $\partial_r$
on $T\Sigma$, which has length denoted by $\dr$. If $\gamma$ is an isometry of $\H^3$, the radial function $s(x)$ induced by this new coordinate system is such that
$$|s(x)-r(x)|\leq C\quad\mbox{for all x }\in M\setminus K,$$
where $C$ depends only on the distance from $\gamma$ to the identity. We denote  by $\partial_s$, $\partial_s^{\top}$, and $|\partial_s^{\top}|$ the correspondent quantities defined with respect to this new coordinate system. 

With respect to the coordinate system induced by $\gamma$, the metric $g$ can be written as
 $$g=ds^2+\sinh^2s \,g_0+h^{\gamma}/(3\sinh s) + P,$$
 where $h^{\gamma}$ is a symmetric $2$-tensor on $S^2$ and 
$$\lvert P\rvert+\lvert D P\rvert +\lvert D^2 P\rvert+\lvert D^3 P\rvert\leq C\exp(-4r)$$
for some constant $C$ depending only on $C_1$ and the distance from $\gamma$ to the identity. If $v$ is such that
$$\gamma^{*}g_0=\exp(2v)g_0,$$
 the relation between $h$ and $h^{\gamma}$ is given by
$$h^{\gamma}=\exp(v)\gamma^* h\quad\mbox{and}\quad \tr h^{\gamma}=\exp(3v)\tr h\circ \gamma.$$

A standard application of Brower's fixed point Theorem 
implies the existence of a conformal diffeomorphism  $\gamma$  such that, for $i=1,2,3,$
\begin{equation*}
 \int_{S^2}x_i\tr h^{\gamma} d\mu_0=0,
\end{equation*}
where $x_i$ denote the coordinate functions for the unit ball in $\R^3$.
Moreover, if $\tr h$ is positive, the computations done in \cite[page 292]{wang} show that such $\gamma$ is unique. 

Finally, 
a surface $\Sigma$ with constant mean curvature is said to be stable if volume preserving variations do not
decrease its area. A standard computation shows that stability is equivalent to the second variation operator
$$ Pf=-\Delta f -\left(\lvert A\rvert^2+R(\nu,\nu) \right)f$$
having only nonnegative eigenvalues when restricted to functions with zero mean value, i.e.,
     $$\int_{\Sigma_t}\left(\lvert A\rvert^2+R(\nu,\nu) \right)f^2d\mu\leq \int_{\Sigma_t}\lvert \nabla f \rvert^2 d\mu$$
for all functions $f$ with $\int_{\Sigma_t}fd\mu=0.$

\subsection{Statement of main result}

From now on, we  assume that $\tr h$ is positive and so we will use the radial function $r(x)$ and all of its associated quantities  to denote the unique coordinate system satisfying  
\begin{equation}\label{centered}
 \int_{S^2}x_i\tr h d\mu_0=0\quad\mbox{for } i=1,2,3.
\end{equation}

We say that an asymptotic hyperbolic $3$-manifold satisfies hypothesis (H) if we can find positive constants $r_1, C_1,C_2,$ and $C_3$ such that, with respect to the coordinate system satisfying \eqref{centered},

$$(H)\qquad\left \{ \begin{aligned}
					&\mbox{$(M,g)$ is asymptotically hyperbolic with constants }r_1\mbox{ and }C_1,\\
					& |h|_{C^3(S^2)}\leq C_2,\\
					& \tr h\geq C_3.
				\end{aligned}
\right.$$

The main purpose of this paper is to show
 \begin{thm} Let $(M,g)$ be an asymptotically hyperbolic manifold satisfying hypothesis (H). Outside a compact set, $M$   admits a foliation by stable spheres with constant mean curvature. The foliation is unique among those with the property that, for some constant $C_4$, each leaf has
  $$\overline r-\underline r\leq C_4.$$
  
  Furthermore, there are  constants $$C=C(C_1,C_2,C_3,C_4, r_1)\quad\mbox{and}\quad r_0=r_0(C_1,C_2,C_3,C_4, r_1)$$ such that each leaf $\Sigma$ with $\underline r\geq r_0$ satisfies the following:
  \begin{itemize}
  \item[(i)] If we set $$w(x)=r(x)-\hat r\quad\mbox{where}\quad|\Sigma|=4\pi\sinh^2 \hat r,$$
then
$$\sup_{\Sigma} |w|\leq C\exp(-\underline r)\quad\mbox{and}\quad\int_{\Sigma}| \trr|^2d\mu\leq C\exp(-2\underline r);$$
  \item[(ii)] $$\int_{\Sigma}\n^2d\mu\leq C\exp(-4\underline{r}); $$
  \item[(iii)] $\Sigma$ can be written as $$\Sigma=\{(\hat r+f(\theta),\theta)\,|\, \theta\in S^2\}\quad\mbox{with}\quad |f|_{C^2(S^2)}\leq C.$$
  \end{itemize}
 \end{thm}

Throughout the rest of this paper we will be using the following notation.  $\Sigma$ will always  denote a stable sphere with constant mean curvature in $(M,g)$ for which
$$\overline r-\underline r\leq C_4.$$
We say that a geometric quantity defined on $\Sigma$ is $T=O(\exp(-nr))$ whenever we can find a constant $C=C(C_1,C_2, C_3, C_4, r_1)$ for which
$$|T|\leq C\exp(-nr).$$

\section{Preliminaries}\label{prelim}

In this section  we adapt the formulas derived in \cite{neves} to the class of asymptotically hyperbolic $3$-manifolds considered in this paper.  $\gamma$ is an isometry of $\H^3$ with induced radial function $s(x)$ and  such that
$$\mbox{dist}(\gamma,\mbox{Id})\leq C_0.$$ 
We denote by $T=O_0(\exp(-nr))$ any quantity defined on $\Sigma$ for which we can find a constant $C=C(C_0,C_1,C_2, C_3, C_4, r_1)$ such that
$$|T|\leq C\exp(-nr).$$
The hyperbolic metric is denoted by $\bar g$ and, with respect to the spherical coordinates $(s,\theta),$ we denote by $e_{\theta}$ any tangent vector to the coordinate spheres with hyperbolic norm one.
\begin{lemm}\label{ads}
$ $
    \begin{enumerate}
        \item[(i)] The mean curvature $H(s)$ of 
        $$\{x\in M\setminus K\,\,|\,\, s(x)=s\}$$ satisfies
            $$H(s)=2\frac{\cosh s}{\sinh s}-\frac{\tr h^{\gamma}}{2\sinh^3 s}+O_0(\exp(-4r)).$$
        \item[(ii)] The Ricci curvature satisfies
            \begin{align*}
                Rc(\partial_s,\partial_s)&=-2-\frac{\tr h^{\gamma}}{2\sinh^3 s}+O_0(\exp(-4r)),\\
                Rc(e_{\theta},e'_{\theta})&=-2\bar g(e_{\theta},e'_{\theta})+O_0(\exp(-3r)),\\
                Rc(e_{\theta},\partial_s)&=O_0(\exp(-4r)).
            \end{align*}
     \item[(iii)]  The scalar curvature satisfies $R(g)=-6+O(\exp(-4r))$.
     \item[(iv)] $$|Rc(\partial_s,\cdot )^{\top}|\leq |\ts|O_0(\exp(-3r))+O_0(\exp(-4r))$$
     and 
     $$Rc(\nu,\nu)=Rc(\partial_s,\partial_s)+|\ts|^2O_0(\exp(-3r))+O_0(\exp(-4r)).$$
     \item[(v)] The Gaussian curvature of $\Sigma$ satisfies
    \begin{multline*}
        K=\frac{H^2-4}{4}+\frac{\tr h^{\gamma}}{2\sinh^3 s}
        +|\ts|^2O_0(\exp(-3r))-  \frac{\n}{2}+O_0(\exp(-4r)).
  \end{multline*}
	\end{enumerate}
\end{lemm}
\begin{proof}

We index the coordinates  $(s,\theta_1,\theta_2)$ by $0, 1$, and $2$ and we assume that 
$(\theta_1,\theta_2)$ are normal coordinates for the metric $g_0$ and that $\partial_{\theta_1}, \partial_{\theta_2}$ are eigenfunctions for $h^\gamma$. We also assume, without loss of generality, that 
\begin{equation}\label{met}
	g=ds^2+\sinh^2 s\,  g_0+h^\gamma/(3\sinh s).
\end{equation}

Then
$$H(s)=-g^{11}\Gamma^r_{11}-g^{22}\Gamma^r_{22}=2\frac{\cosh s}{\sinh s}-\frac{\tr h^{\gamma}}{2\sinh^3 s}+O_0(\exp(-4r)).$$

With respect to these coordinates 
$$Rc(\partial_s,\partial_s)=R^j_{jss}=\partial_j \Gamma_{ss}^j-\partial_s \Gamma_{js}^j+\Gamma_{ss}^m\Gamma_{jm}^j-
\Gamma^{m}_{js}\Gamma^{j}_{sm}.$$ 
We have that
$$\Gamma_{ss}^j=0,\quad\mbox\quad\Gamma_{s1}^2=0,$$
and thus
\begin{multline*}\Gamma_{ss}^m\Gamma_{jm}^j-
\Gamma^{m}_{js}\Gamma^{j}_{sm}=-\left(\Gamma_{1s}^1\right)^2-\left(\Gamma_{2s}^2\right)^2=-\sum_{i=1}^2\left(\frac{\partial_s g_{ii}}{2g_{ii}}\right)^2\\
=-\sum_{i=1}^2\left(\frac{\cosh r}{\sinh r}-\frac{h^{\gamma}_{ii}}{2\sinh^3 r}+O_0(\exp(-4r))\right)^2.
\end{multline*}
On the other hand,
\begin{multline*}
\partial_j \Gamma_{ss}^j-\partial_s \Gamma_{js}^j=-\sum_{i=1}^2\partial_s \Gamma_{is}^i\\
=-2-\frac{2\cosh^2 s}{\sinh^2 s}+\frac{\tr h^{\gamma}}{2\sinh s}+\sum_{i=1}^2 2\left(\Gamma_{is}^i\right)^2+O_0(\exp(-4r))
\end{multline*}
and this implies the first identity of (ii). The other two identities and property (iii) can be checked in the same way.

We now prove the first identity in (iv). Decompose the vector $\nu$ as
$$\nu=a\partial_s+\bar \nu$$ and a tangent vector $X$ as
$$X=b\partial_s+\bar X,$$ where both $\bar \nu$ and $\bar X$ are tangent to coordinate spheres. Hence, due to \eqref{met},
$$1=a^2+|\bar \nu|^2,\quad |X|^2=b^2+|\bar X|^2$$
and 
$$|\bar \nu|^2=|\ts|^2,\quad b^2\leq |X|^2|\ts|^2.$$
Therefore, because
$$0=ab +\bar g(\bar \nu,\bar X)+|\bar X||\bar \nu|O_0(\exp(-3r)),$$
we obtain
\begin{align*}
|Rc(\nu,X)|=& |abRc(\partial_s,\partial_s)+aRc(\partial_s,\bar X)+bRc(\partial_s,\bar \nu)+Rc(\bar \nu,\bar X)|\\
\leq & |2ab +2\bar g(\bar \nu,\bar X)|+|\langle\nu,\partial_s\rangle\langle X,\partial_s\rangle| O_0(\exp(-3r))\\
&+ O_0(\exp(-4r))+|\bar X||\bar \nu|O_0(\exp(-3r))\\
\leq & |X||\ts|O_0(\exp(-3r))+O_0(\exp(-4r)).
\end{align*}
The other identity can be checked in the same way.

Finally, property (v) follows from  Gauss equation
$$K=R/2-R(\nu,\nu)+H^2/4-\lvert \mathring{A}\rvert^2 /2.$$
\end{proof}

For the following lemma, we use an observation due to Metzger \cite{metzger} which amounts to not applying Leibniz rule to the terms involving $R_{\nu ijk}$ in Simmon's identity.
\begin{lemm}\label{A}
    The Laplacian of $|\tf|^2$ satisfies
    \begin{multline*}
         \Delta \left(\frac{{\lvert\mathring{A}\rvert^2}}{2}\right) = \left(\frac{H^2-4}{2}-
         \lvert \tf \rvert ^2+O(\exp(-3r))\right)\lvert \tf \rvert ^2+\lvert\nabla\tf\rvert^2 \\
         + \tf_{ij}\nabla_k(R_{\nu jik})+\tf_{ij}\nabla_i(Rc(\nu,\partial_j)).
         \end{multline*}
Moreover, integration by parts implies
$$\int_{\Sigma}\tf_{ij}\nabla_k(R_{\nu jik})+\tf_{ij}\nabla_i(Rc(\nu,\partial_j)) d\mu=-2\int_{\Sigma}|Rc(\nu,\cdot)^{\top}|^2d\mu.$$
\end{lemm}

\begin{proof}

    We assume normal coordinates $x=\{x^i\}_{i=1,2}$ around a point $p$ in the constant mean curvature surface
    $\Sigma$. The tangent vectors are denoted by $\{\partial_1,\partial_2\}$ and the
    Einstein summation convention for the sum of repeated indices is used.

    Simons' identity for the Laplacian of the second fundamental form $A$ (see for instance
    \cite{Huisken1}) implies that
     \begin{multline*}
        \Delta \tf = \left(\frac{H^2}{2}-\n^2 \right)\tf +{R}_{kikm}\tf_{mj}+{R}_{kijm}\tf_{km}
       + H(\tf)^2\\-\frac{\n^2}{2}Hg_{ij} +\nabla_k(R_{\nu jik})+\nabla_i(Rc(\nu,\partial_j)).
        \end{multline*}
        Because $\langle(\tf)^2,\tf\rangle=0$, we have
        \begin{multline*}
        \Delta  \left(\frac{{\lvert\mathring{A}\rvert^2}}{2}\right) = \left(\frac{H^2}{2}-\n^2 \right)\n^2 +{R}_{kikm}\tf_{mj}\tf_{ij}+{R}_{kijm}\tf_{km}\tf_{ij}
       \\+|\nabla\tf |^2+\tf_{ij}\nabla_k(R_{\nu jik})+\tf_{ij}\nabla_i(Rc(\nu,\partial_j)).
        \end{multline*}
   The desired formula follows  from
    \begin{equation*}
        {R}_{stuv}=-(\delta_{su}\delta_{tv}-\delta_{sv}\delta_{tu})+O_0(\exp({-3r})).
    \end{equation*}
 
\end{proof}

Finally, we derive the equation for the Laplacian of $s(x)$ on $\Sigma$. 
\begin{prop}\label{laplace}
The Laplacian of $s$ on $\Sigma$ satisfies
\begin{multline*}
    \Delta s = (4-2\lvert\partial _s^{\top} \rvert^2 )\exp(-2s)+2-H\\ +(H-2)(1-\langle \partial _s, \nu \rangle)+(1-\langle \partial _s, \nu \rangle)^2
        + O_0(\exp(-3r))
\end{multline*}
or, being more detailed,
\begin{multline*}
    \Delta s = H(s)-H \\+(H-2)(1-\langle \partial _s, \nu \rangle)
    +(1-\langle \partial _s, \nu \rangle)^2
        -2\ds^2\exp(-2s)\\+\ds^2O_0(\exp{(-3r)})+O_0(\exp(-4r)).
\end{multline*}
\end{prop}
\begin{proof}
 Assume that, without loss of generality,  
$$g=ds^2+\sinh^2 s\,  g_0+h^\gamma/(3\sinh s).$$

It suffices to check that
$$ \d_{\Sigma}\partial_s=H(s)-\frac{H(s)}{2}|\ts|^2+|\ts|^2O_0(\exp(-3r))$$
because the rest of the proof follows exactly as in \cite[Proposition 3.14]{neves}.

Given a point $p$ in $\Sigma$, consider an orthonormal frame $ e_1,e_2$ to $\Sigma$ such that $ e_1$ is tangent to coordinate spheres and 
$$e_2=a\partial_s+\bar e_2,\quad i=1,2,$$
where the vector $\bar e_2$ is tangent to coordinate spheres. Denote by $\bar A(s)$ is the second fundamental form for the coordinate sphere $\{|x|=s\}$ and by  $\bar A_{22}(s)$  its evaluation on the unit vector $|\bar e_2|^{-1}\bar e_2.$
Then
$$\langle D_{e_2} \partial_s, e_2 \rangle=\bar A(\bar e_2,\bar e_2)$$
 and thus
$$
 \d_{\Sigma}\partial_s=H(s)-|\ts|^2\bar A_{22}=H(s)-\frac{H(s)}{2}|\ts|^2+|\ts|^2O_0(\exp(-3r)).
$$
   
\end{proof}

\section{Integral estimates}\label{intestimate}

We keep assuming that $\gamma$ is an isometry of $\H^3$ with induced radial function $s(x)$ and  such that
$$\mbox{dist}(\gamma,\mbox{Id})\leq C_0.$$ 
We denote by $T=O_0(\exp(-nr))$ any quantity defined on $\Sigma$ for which we can find a constant $C=C(C_0,C_1,C_2, C_3, C_4, r_1)$ such that
$$|T|\leq C\exp(-nr).$$

We follow \cite{Huisken3}  and use the stability condition in order to estimate the mean curvature
$H$ and the $L^2$ norm of $\n$. 

\begin{lemm}\label{h}
    We have that
        $$H^2=4+\frac{16\pi}{\lvert\Sigma\rvert}+\fint_{\Sigma}O(\exp(-3r))d\mu$$
    or, equivalently,
        $$H=2+\frac{4\pi}{\lvert\Sigma\rvert}+\fint_{\Sigma}O(\exp(-3r))d\mu. $$
        Moreover,
        $$\fint_{\Sigma}\n^2d\mu=\fint_{\Sigma}O(\exp(-3r))d\mu.$$
\end{lemm}
The proof is the same as in \cite[Lemma 4.1]{neves}.

Proposition 4.2 of \cite{neves} can be easily adapted to show

\begin{prop}\label{estimates}
 The following identities hold:
    \begin{enumerate}
        \item[(i)]
            $$\int_{\Sigma}\exp{(-2s)}d\mu=\pi+O_0(\exp(-\underline{r})).$$
        In particular, there is a constant $C=C(C_0,C_1,C_2, C_3, C_4,r_1)$ so that $$C^{-1}\exp(2\underline{s})\leq |\Sigma| \leq C\exp(2\overline{s}).$$
        \item[(ii)]$$\int_{\Sigma}(1-\langle\partial_s,\nu\rangle)^2d\mu=O_0(\exp(-\underline{r})).$$
        \item[(iii)]
                    $$\int_{\Sigma}\ds^2 d\mu = O_0(1).$$
    \end{enumerate}
\end{prop}
\begin{proof}
The first two properties were shown in  \cite[Proposition 4.2]{neves}. We only need to show the last property.

    Integrating by parts in Proposition \ref{laplace} and using Lemma \ref{h} we obtain
     
        \begin{multline*}
            \int_{\Sigma}\ds^2 (1-2\exp(-2s)(s-\underline s))d\mu=-{\int}_{\Sigma}4\exp(-2s)(s-\underline s)d\mu\\
            +4\pi{\fint}_{\Sigma} (s-\underline s) d\mu-{\fint}_{\Sigma}4\pi(1-\langle\partial_s,\nu\rangle)(s-\underline s)d\mu
             \\-\int_{\Sigma}(1-\langle\partial_r,\nu\rangle)^2 (s-\underline s)d\mu+\int_{\Sigma}(s-\underline s)O_0(\exp(-3r))d\mu
             \\+\fint_{\Sigma}O_0(\exp(-3r))d\mu\int_{\Sigma}(s-\underline s)d\mu.
        \end{multline*}
   We know that   $\overline s-\underline s\leq C$
   for some $C=C(C_0,C_4)$ and so we can find $r_0$ such that for all $\underline r\geq r_0$ we have
   $$\int_{\Sigma}\ds^2 d\mu \leq\frac{1}{2}\int_{\Sigma}\ds^2 (1-2\exp(-2s)(s-\underline s))d\mu.$$
    The third property follows at once.
 \end{proof}

The stability of $\Sigma$ can be used in the same way as in \cite[Section 5]{Huisken3} in order to obtain
the next proposition.

\begin{prop}\label{integral}
There is $r_0=r_0(C_1,C_2, C_3,C_4, r_1)$ such that if $\underline{r}\geq r_0$ the following estimate holds
    \begin{equation*}
            \fint_{\Sigma}\n^2d\mu+\int_{\Sigma}\n^4d\mu+\int_{\Sigma}\lvert\nabla
            \tf\rvert^2d\mu\leq \bar C\int_{\Sigma} |Rc(\nu,\cdot)^{\top}|^2d\mu,
    \end{equation*}
  where $\bar C$ is a universal constant. 
In particular,
$$\int_{\Sigma}\n^2d\mu\leq O(\exp(-4\underline{r})). $$
\end{prop}
\begin{proof}
From Proposition \ref{estimates} we have that
$$|\Sigma|\exp(-3\underline r)= O(\exp(-\underline r)$$
and thus, integrating the identity in Lemma \ref{A} and using Lemma \ref{h}, we can choose $r_0$ so that for $\underline r\geq r_0$
        \begin{multline}\label{f1}
            7\pi\fint_{\Sigma}\n^2d\mu+\int_{\Sigma}\lvert\nabla \tf\rvert^2d\mu\leq\int_{\Sigma}\n^4d\mu+2\int_{\Sigma} |Rc(\nu,\cdot)^{\top}|^2d\mu\\        
            \end{multline}
    We can further choose $r_0$ so that $H\geq 2$ whenever $\underline r\geq r_0$. For that reason
        $$\lvert A \rvert^2 +R(\nu,\nu) \geq \n^2+O(\exp(-3r)),$$
        and so the  stability assumption implies that
            $$\int_{\Sigma}\n^2 f^2d\mu\leq \int_{\Sigma}\lvert \nabla f \rvert^2 d\mu+O(\exp(-3\underline r))\int_{\Sigma}f^2d\mu$$
    for all functions $f$ with $\int_{\Sigma}fd\mu=0.$        
    Following the same argumentation as  in \cite[Proposition 4.3]{neves}, we get that for all $\varepsilon >0$ we can find $\bar C=\bar C(\varepsilon)$ for which
        \begin{multline*}
            (1-\varepsilon)\int_{\Sigma}\n^4d\mu\leq
            \frac{1}{2-\varepsilon}\int_{\Sigma}\lvert\nabla\tf\rvert^2d\mu+O(\exp(-\underline r))\fint_{\Sigma}\n^2d\mu\\
            +\bar C \int_{\Sigma} |Rc(\nu,\cdot)^{\top}|^2 d\mu.
        \end{multline*}
    There is $r_0$ so that, for all $\underline r\geq r_0$, we can multiply this inequality by $(1+\varepsilon)/(1-\varepsilon)$ (with $\varepsilon$ small) and add to equation \eqref{f1} in order to obtain
        \begin{equation*}
            \fint_{\Sigma}\n^2d\mu+\int_{\Sigma}\n^4d\mu+\int_{\Sigma}\lvert\nabla \tf\rvert^2d\mu\leq \bar C \int_{\Sigma} |Rc(\nu,\cdot)^{\top}|^2 d\mu
        \end{equation*}
        for some universal constant $\bar C$.
        
        The last assertion  is a consequence of Lemma \ref{ads} and Proposition \ref{estimates}.
\end{proof}

\section{Intrinsic geometry}\label{Intrin}

We  continue adapting the work done by the authors in \cite{neves} and so we now study the intrinsic geometry of $\Sigma$. More precisely, we show

\begin{thm}\label{intrinsic}
There is $r_0=r_0(C_1,C_2, C_3,C_4, r_1)$ so that if $\underline r\geq r_0$ the following property holds. After pulling back by a suitable diffeomorphism from $\Sigma $ to $S^2$, the metric
    $$\hat{g} =4\pi\a ^{-1}{g }$$
can be written as $$\exp(2\beta )g_0$$ with
    $$\sup\lvert\beta \rvert=O(\exp(-\underline{r} )),\quad
    \int_{S^2}|\nabla \beta |^2d\mu_0=O(\exp(-2\underline{r} )),$$
    and 
    \begin{equation}\label{c1}
       \int_{S^2}x_j \exp{(2\beta )}d{\mu}_0 =0\quad\mbox{for }\,j=1,2,3,
    \end{equation}
     where the norms are computed with respect to $g_0$, the standard round metric on $S^2$.
\end{thm}

Like in \cite[Theorem 5.1]{neves}, we need to estimate $\n$. The theorem will then  follow from Gauss equation. Therefore, we start by proving

\begin{thm}\label{2ff}
There are constants $$r_0=r_0(C_1,C_2, C_3,C_4, r_1)\quad\mbox{and}\quad C=C(C_1,C_2, C_3,C_4, r_1)$$ such that, provided $\underline r\geq r_0$, we have
$$\n^2\leq C\exp(-4\underline r).$$
\end{thm}
\begin{proof}

We will  need the following consequence of Lemma \ref{A}.
\begin{lemm}\label{a2}
There are constants $$r_0=r_0(C_1,C_2, C_3,C_4, r_1)\quad\mbox{and}\quad C=C(C_1,C_2, C_3,C_4, r_1)$$ such that, provided $\underline r\geq r_0$, we have
$$ \Delta \left(\frac{{\lvert\mathring{A}\rvert^2}}{2}\right)\geq -\n^4-C\n^2-C\exp(-6r) .$$

\end{lemm}
\begin{proof}
	A simple computation shows that
	\begin{align*}
	\nabla_k(R_{\nu jik})+\nabla_i(Rc(\nu,\partial_j))= & D_k R_{\nu j ik}+D_{i}Rc({\nu }, \partial_j)+\tf_{km}R_{mjik}-\tf_{ik} R_{\nu j\nu k}\\
	&+\tf_{im}Rc(\partial_m,\partial_j)+\frac{H}{2}R_{\nu j \nu i}-A_{ij}Rc(\nu,\nu)\\
	=& O(\exp(-3r))+\tf_{ij}O(\exp(-3r))
	+\frac{H}{2}R_{\nu j \nu i}
	\\&-\frac{H}{2}g_{ij}Rc(\nu,\nu).
	\end{align*}
	Moreover,
	$$|\tf_{ij}R_{\nu j \nu i}|\leq \n O(\exp(-3r))$$
	and thus, choosing $r_0$ so that for all $\underline r\geq r_0$ we have $ H\leq 3$, we obtain that
	$$\tf_{ij}\nabla_k(R_{\nu jik})+\tf_{ij}\nabla_i(Rc(\nu,\partial_j))\geq -C\n^2-C\exp(-6r)$$
	for some $C=C(C_1,C_2, C_3,C_4, r_1).$ Lemma \ref{A} implies the desired result.
	
\end{proof}

Next, we argue that we can choose $r_0$ so that for all $\underline r\geq r_0$ 
$$\n^2\leq C_B,$$
where $C_B=C_B(C_1,C_2, C_3,C_4, r_1).$
 Suppose that $$\sup_{\Sigma}\n=\n(x_1)\equiv 1/\sigma\quad\mbox{with }\sigma\leq \varepsilon_0$$
        where $\varepsilon_0$ will be chosen later (and depending only on $C_1,C_2, C_3,C_4,$ and  $r_1$). 
        
        Set ${g_{\sigma}}=\sigma^{-2}\,g$ and denote the various geometric quantities
    with respect to $g_{\sigma}$ using an index $\sigma$. We can take $\varepsilon_0$ and $r_0$ so that for all $\underline r\geq r_0$ the mean curvature with respect to $g_{\sigma}$ satisfies $H_{\sigma}=\sigma H\leq 1$. Therefore     $${\lvert A\rvert}^2_{\sigma}=H_{\sigma}^2/2+\n^2_{\sigma}=\sigma^2 (H^2/2+\n^2)\leq 2.$$
    
   The argumentation in the proof of \cite[Lemma 5.2]{neves} shows that, provided we fix $\varepsilon_0$ sufficiently small, there are uniform constants $s_0$ and $C_S$ so that
    for every compactly supported function $u$
        $$\left({\int_{{B}^{\sigma}_{s_0}(x_1)\cap\Sigma}u^2\,d\mu_{\sigma}}\right)^{1/2}\leq
        2C_S\int_{{B}^{\sigma}_{s_0}(x_1)\cap\Sigma}\lvert{\nabla}u\rvert_{\sigma}\,d\mu_{\sigma}.$$
    Finally, because $\n_{\sigma}$ is uniformly bounded, it follows easily from Lemma \ref{a2} that
        $$\Delta_{\sigma}\n_{\sigma}^2\geq -C\n_{\sigma}^2-C\exp(-6r),$$
        where $C=C(C_1,C_2, C_3,C_4, r_1).$
    We have now all the necessary conditions to apply Moser's iteration argument (see, for instance, \cite[Lemma 11.1.]{Li}) and obtain that, for some
    constant $C=C(C_1,C_2, C_3,C_4, r_1)$,
        \begin{multline*}
            1=\n_{\sigma}^2(x_1)\leq C \int_{\Sigma}\n_{\sigma}^2d\mu_{\sigma}+C\exp(-6r)\\
            =C\int_{\Sigma}\n^2d\mu+C\exp(-6\underline r)=C\exp(-4\underline r).
        \end{multline*}
	This gives us a contradiction if we choose $r_0$ sufficiently large. 
	
	Because $\n$ is uniformly bounded provided we choose $r_0$ sufficiently large, we can argue again as in the proof of  \cite[Lemma 5.2]{neves}  and conclude the existence of uniform constants $s_0$ and $C_S$ so that
    for every compactly supported function $u$ and $x\in \Sigma$
        $$\left({\int_{{B}_{s_0}(x)\cap\Sigma}u^2\,d\mu}\right)^{1/2}\leq
        2C_S\int_{{B}_{s_0}(x)\cap\Sigma}\lvert{\nabla}u\rvert\,d\mu.$$
       Moreover
       $$ \Delta \left(\frac{{\lvert\mathring{A}\rvert^2}}{2}\right)\geq -(C_B^2+C)\n^2-C\exp(-6r)$$
       and so Moser's iteration implies that
       $$\n^2\leq C \int_{\Sigma}\n^2d\mu+C\exp(-6r)\leq C\exp(-4\underline r),$$
       where $C=C(C_1,C_2, C_3,C_4, r_1).$
\end{proof}

We can now prove Theorem \ref{intrinsic}.

\begin{proof}[{\bf Proof of Theorem \ref{intrinsic}}]
Denote by $ \widehat K$ the Gaussian curvature of $\hat g$. Because $ \widehat K= |\Sigma| K (4\pi)^{-1}$, we obtain from Lemma \ref{ads}, Lemma \ref{h}, and Theorem \ref{2ff} that, provided $r_0$ is sufficiently large,
$$| \widehat K-1|=O(\exp(-\underline r)).$$

Standard theory (see for instance the proof of \cite[Theorem 5.1]{neves}) implies that if we choose $r_0$ large enough so that $| \widehat K-1|$ is sufficiently small for all $r\geq r_0$, then, after pulling back by a diffeomorphism, the metric
    $\hat{g} $ can be written as $\exp(2\beta )g_0$ where the smooth function $\beta$ satisfies all the desired conditions
\end{proof}

\section{Approximation to coordinate spheres}\label{ap}

In this section we show that a stable sphere $\Sigma$ with constant mean curvature is close  to some coordinate spheres and that the corresponding coordinate system is at a bounded distance from the identity.  Like in \cite{metzger}, this result relies on a theorem by De Lellis and M\"uller \cite{DeLellis}.

\begin{thm}\label{Aprox} There are constants $r_0, C_0,$ and $C$ depending only on $C_1,C_2, C_3,C_4,$ and $r_1$ for which, if $\underline r\geq r_0$, the following property holds.

There is an isometry $\gamma$ of $\H^3$ with 
$$\mbox{dist}(\gamma,\mbox{Id})\leq C_0$$
such that,
if we consider the function on $\Sigma$ given by
$$u(x)=s(x)-\hat r\quad\mbox{where}\quad|\Sigma|=4\pi\sinh^2 \hat r,$$
then
$$\sup_{\Sigma} |u|\leq C\exp(-\underline r)\quad\mbox{and}\quad\int_{\Sigma}|\ts|^2d\mu\leq C\exp(-2\underline r).$$
\end{thm}
\begin{proof}
 Fix an isometry between hyperbolic space and the unit ball
 $$F \colon M\setminus K\longrightarrow B_1\setminus \mbox{Ball}.$$
 Denoting the euclidean distance by $|x|$ and the hyperbolic induced measure on $\Sigma$ by $d\bar \mu$, we know that
 $$\sinh r=\frac{2|x|}{1-|x|^2}$$
 and $d\bar\mu-d\mu=O(\exp(-3r))d\mu$ (see \cite[Section 7]{Huisken2}), respectively.

 Let $\hat \tf$ and $\bar \tf$ denote, respectively, the trace free part of the second fundamental form with respect to the euclidean metric and hyperbolic metric. De Lellis and M\"uller Theorem \cite{DeLellis} implies the existence of a universal constant $C_U$ such that
 \begin{equation*} \sup_{\Sigma}||x-\vec a|-R|^2\leq C_U R^2 \int_{\Sigma}\left|\hat \tf\right|^2d\mathcal{H}^2=C_UR^2\int_{\Sigma}\left|\bar \tf\right|^2d\bar \mu,
\end{equation*} 
where $R$ and $\vec a$ are defined as
 $$R^2=(4\pi)^{-1}\mathcal H^2(\Sigma)\quad\mbox{and}\quad \vec a=\fint_{\Sigma} \mbox{id}_{\Sigma}\,d\mathcal{H}^2.$$

 On the other hand, we know  from Proposition \ref{estimates} that
 \begin{multline*}
 \mathcal{H}^2(\Sigma)=\int_{\Sigma}\frac{(1-|x|^2)^2}{4}d\bar \mu=\int_{\Sigma}\frac{|x|^2}{\sinh^2 r}d\mu+O(\exp(-3\underline r))\\
 =\int_{\Sigma}\sinh^{-2} rd\mu+O(\exp(-\underline r))=\int_{\Sigma}{4}\exp(-2r)d\mu+O(\exp(-\underline r))\\
 =4\pi+O(\exp(-\underline r))
 \end{multline*}
 and thus
 $$R=1+O(\exp(-\exp(-\underline)).$$
 Moreover, we also know from \cite[Section 7]{Huisken2} that
 $$\bar \tf= \tf+AO(\exp(-3r))+O(\exp(-3r)$$
 and thus, it follows from Lemma \ref{h} and Theorem \ref{2ff} that, if we take $r_0$ sufficiently large,
 $$\left|\bar \tf\right|^2\leq 2|\tf|^2+O(\exp(-6\underline r)).$$
  As a result, we obtain from Proposition \ref{estimates}
 $$\int_{\Sigma}\left|\bar \tf\right|^2d\bar \mu\leq 2\int_{\Sigma}\n^2d\mu+O(\exp(-4\underline r))\leq O(\exp(-4\underline r)).$$
 
 Therefore, we have
 \begin{equation}\label{silence}
 \sup_{\Sigma}||x-\vec a|-R|\leq D_1\exp(-2\underline r),
 \end{equation}
where $D_1=D_1(C_1,C_2, C_3,C_4, r_1)$. 

 There is a constant $D_2=D_2(r_1,C_4)$ such that, for all $x$ in $\Sigma$,
\begin{equation}\label{sil}
D_2 \exp(-\underline r)\geq 1-|x|\geq D_2^{-1}\exp(-\underline r).
\end{equation}
Denote an hyperbolic geodesic ball of radius $s$ around a point $p$ by $\bar B_s(p)$. Consider $\bar B_{\hat s}(p)\subset M\setminus K$ such that
$$F(\bar B_{\hat s}(p))=B_{R}(\vec a).$$ Inequalities \eqref{silence} and \eqref{sil} combined with  standard facts in hyperbolic geometry imply the existence of some constant $D_3=D_3(D_1,D_2,r_1)$ for which
\begin{equation}\label{prox}
	\sup_{\Sigma}|\mbox{dist}_{\H^3}(x,p)-\hat s|\leq D_3\exp(-\underline r).
\end{equation}

Define $\gamma$ to be an isometry of $\H^3$ for which $\gamma(\bar B_{\hat s}(0))=\bar B_{\hat s}(p)$. The equation above and the fact that $\bar r-\underline r\leq C_4$ implies that, after choosing $r_0$ large enough, we have for all $\underline r\geq r_0$
$$\{x\in \Sigma\,|\,\mbox{dist}_{\H^3}(x,p)=\hat s\}\subset \bar B_{\overline r+D_4}(0)\setminus \bar B_{\underline r-D_4}(0)$$
 for some $D_4=D_4(C_4,D_3, r_1)$ and thus
$$\mbox{dist}(\gamma,\mbox{Id})\leq C_0,$$ 
where $C_0=C_0(C_1,C_2, C_3,C_4, r_1)$. 

From \eqref{prox} we know that, with respect to the coordinate system induced by $\gamma$, 
$$\sup_{\Sigma}|s(x)-\hat s|=O(\exp(-\underline r))$$
and so we can apply Proposition \ref{estimates} (i) to conclude 
$$\int_{\Sigma}\frac{4}{\exp(2\hat s)}d\mu=4\pi+O(\exp(-\underline r))\implies |\Sigma|=4\pi\sinh^2 \hat s+O(\exp(-\underline r)).$$ This implies that the function $u$ defined on the statement of Theorem \ref{aprox} satisfies $|u|=O(\exp(-\underline r))$ and that
$$2+\frac{4\pi}{|\Sigma|}=\frac{2\cosh \hat r}{\sinh \hat r}+O(\exp(-3\underline r))=\frac{2\cosh  s(x)}{\sinh  s(x)}+O(\exp(-3\underline r))$$
for all $x$ in $\Sigma$. As a result, we have from Lemma \ref{ads} (i) and Lemma \ref{h} that
$$H(s(x))=H+O(\exp(-3\underline r))\quad\mbox{for every } x\in \Sigma.$$
We can now integrate the second identity in Proposition \ref{laplace} against the function $u$ and use Proposition \ref{estimates} in order to obtain
 $$\int_{\Sigma}|\ts|^2d\mu\leq O(\exp(-2\underline r)).$$
\end{proof}

\section{Unique approximation to coordinate spheres}\label{Unique}

The purpose of this section is to show that the constant mean curvature stable sphere $\Sigma$ must be very close to the coordinate spheres induced by our fixed coordinate system, i.e., the one satisfying condition \eqref{centered}.
Such result is obviously false in hyperbolic space and so we need to use the fact that the ambient manifold is an asymptotically hyperbolic
space with $\tr h>0$. Like in \cite{neves}, this will be accomplished using  the Kazdan-Warner identity. More precisely, we show

\begin{thm}\label{sil3} There are constants $r_0$ and $C$ depending only on $C_1,C_2, C_3,C_4,$ and $r_1$ for which, if $\underline r\geq r_0$, the  function on $\Sigma$ given by
$$w(x)=r(x)-\hat r\quad\mbox{where}\quad|\Sigma|=4\pi\sinh^2 \hat r$$
satisfies
$$\sup_{\Sigma} |w|\leq C\exp(-\underline r)\quad\mbox{and}\quad\int_{\Sigma}| \trr|^2d\mu\leq C\exp(-2\underline r).$$

Moreover, $\Sigma$ can be written as $$\Sigma=\{(\hat r+f(\theta),\theta)\,|\, \theta\in S^2\}\quad\mbox{with}\quad |f|_{C^2(S^2)}\leq C.$$
\end{thm}

\begin{proof}
Denote by $ \widehat K$ the Gaussian curvature of $\hat g$, the metric defined on  Theorem \ref{intrinsic}.
 From the Kazdan-Warner identity \cite{Kazdan} we know that, for each $i=1,2,3,$
 $$\int_{S^2}\langle\nabla \widehat{K} ,\nabla x_i\rangle\exp(2\beta )d{\mu_0}=0$$
 or, equivalently,
 $$\int_{S^2}x_i\widehat{K} \exp(2\beta )d{\mu_0}-\int_{S^2}\widehat K \langle\nabla \beta ,\nabla x_i
 \rangle\exp(2\beta )d{\mu_0}=0.$$
 Theorem \ref{intrinsic} and the fact that
 $$\widehat K =1+O(\exp( -\underline{r} ))$$ 
 implies that
\begin{multline*}
\int_{S^2}\widehat K \langle\nabla \beta ,\nabla x_i
 \rangle\exp(2\beta )d{\mu_0}=\int_{S^2}\langle\nabla \beta ,\nabla x_i\rangle d{\mu_0}+O(\exp(-2\underline r))\\
 =2\int_{S^2}\beta x_id{\mu_0}+O(\exp(-2\underline{r} ))=O(\exp(-2\underline r))
\end{multline*}
and so, the Kazdan-Warner identity becomes
\begin{equation}\label{brubaker}
\int_{S^2} x_i{ \widehat K}\exp(2\beta) d{\mu_0}=O(\exp( -2\underline{r} ))\quad\mbox{for }i=1,2,3.
\end{equation}

From Lemma \ref{ads} (v) and Theorem \ref{2ff} we know that
    \begin{multline*}
     4\pi\widehat{K}   = \a \frac{H^2-4}{4}+|\Sigma|\frac{\tr h^{\gamma} }{2\sinh^3 s}
     - |\ts|^2O(\exp(-\underline r))+O(\exp(-2\underline r))\\
    \end{multline*}
and hence, because $\Sigma $ has constant mean curvature, we obtain from \eqref{brubaker}, Theorem \ref{intrinsic}, and Theorem \ref{Aprox} that
\begin{multline*}
        \int_{S^2}x_i\,\frac{\a ^{3/2}}{\sinh^3 s}\tr h^{\gamma} \exp(2\beta)\,d{\mu_0} = \fint_{\Sigma }\ds^2 O(1)d\mu_0+O(\exp(-\underline r))\\
        =O(\exp(-\underline r)).
    \end{multline*}
On the other hand, from Theorem \ref{intrinsic}, and Theorem \ref{Aprox}, 
$$\frac{\sinh^3 \hat r}{\sinh^3 s}=1+O(\exp(-\underline r))\quad\mbox{and}\quad\exp(2\beta)=1+O(\exp(-\underline r))$$
which implies that
$$\int_{S^2}x_i\,\tr h^{\gamma}\,d{\mu_0}=O(\exp(-\underline r))\quad\mbox{for }i=1,2,3.$$    

We can assume without loss of generality that the isometry $\gamma$ induces the following conformal transformation of $S^2$ (see \cite[page 292]{wang})
$$\gamma^{-1}(x)=\left(\frac{x_1}{\cosh t+x_3\sinh t},\frac{x_2}{\cosh t+x_3 \sinh t},\frac{\sinh t+x_3\cosh t}{\cosh t+x_3\sinh t}\right),$$
where $t$ is the parameter we want to estimate. In this case a direct computation reveals that
$$\left(\gamma^{-1}\right)^* g_0=\exp(2u)g_0\quad\mbox{with}\quad \exp(u)=(\cosh t+x_3 \sinh t)^{-1}.$$
According to Section \ref{sec2} we have that
$$ \tr h^{\gamma}=\exp(-3u\circ\gamma)\tr h\circ \gamma$$
and so, due to 

\begin{equation*}
 \int_{S^2}x_3\tr h d\mu_0=0,
\end{equation*}
we obtain
\begin{multline*}
\int_{S^2}x_3\,\tr h^{\gamma}\,d{\mu_0}=\int_{S^2}x_3\,\exp(-3u\circ\gamma)\tr h\circ \gamma\,d{\mu_0}\\
=\int_{S^2}x_3\circ \gamma^{-1}\exp(-u)\tr hd{\mu_0}=\sinh t \int_{S^2}\tr hd{\mu_0}.
\end{multline*}
As a result, the parameter $t$ has order $\exp(-\underline r)$ and this implies that
$$\mbox{dist}(\gamma,\mbox{Id})\leq C\exp(-\underline r)$$ 
for some $C=C(C_1,C_2, C_3,C_4, r_1).$
An immediate consequence is that
$$\sup_{\Sigma} |w|\leq C\exp(-\underline r)$$
for some $C=C(C_1,C_2, C_3,C_4, r_1).$
Arguing like in the proof of Theorem \ref{Aprox}, it is straightforward to see that
$$\int_{\Sigma}| \trr|^2d\mu= O(\exp(-2\underline r)).$$

Decompose the normal vector $\nu$ as $\nu=a\partial_r+\bar \nu$, where $\bar \nu$ is orthogonal to $\partial_r$ with respect to the hyperbolic metric. In this case,
$$1=a^2+|\bar \nu|^2+O(\exp(-4\underline r))\implies |\bar \nu|\leq |\trr|+O(\exp(-2\underline r)).$$
Thus, assuming normal coordinates at a given point, we have from Theorem \ref{2ff}
\begin{align*}
\lvert \nabla_{\partial_i} \langle \nu, \partial _r \rangle \rvert =& |A(\partial_i,\trr)+\langle\nu,\partial_r\rangle\langle D_{\partial_i},\partial_r\rangle+\langle D_{\partial_i}\partial_r,\bar \nu\rangle|\\
\leq & C|\trr|+C\exp(-2\underline r),
\end{align*}
where $C=C(C_1,C_2, C_3,C_4, r_1)$.
Combining Proposition \ref{laplace}  with the Bochner formula for $\lvert\nabla r\rvert^2$ we obtain that, assuming  $\underline r\geq r_0$ for some $r_0$ chosen sufficiently large,
     \begin{align*}
        \Delta  \dr^2 & \geq -C\dr^2+2\lvert\nabla\nabla r\rvert^2-C\lvert\nabla\nabla r\rvert \dr^2\exp(-2r)-C\exp(-3r)\\
        & \geq -C\dr^2-C\exp(-3r),
    \end{align*}
    where $C=C(C_1,C_2, C_3,C_4, r_1)$. In view of this equation, we can apply  Moser's iteration (like we did in the proof of Theorem \ref{2ff}) in order to conclude that
           $$\sup_{\Sigma }\dr^2\leq C\int_{\Sigma }\dr^2d\mu+C\exp(-3\underline r)\leq C\exp(-2\underline r),$$
    where $C=C(C_1,C_2, C_3,C_4, r_1).$

    Therefore, provided we choose $r_0$ sufficiently large, we have that $ \langle \nu, \partial _r \rangle$ is positive whenever $\underline r \geq r_0$ and this implies that $\Sigma$ is the graph of a function $f$ over the coordinate sphere $\{|x|=\hat r\}.$  Because  
    $$\n^2=O(\exp(-4\underline r), \quad |H^2-4|=O(\exp(-2\underline r)), \quad |w|=O(\exp(-\underline r)),$$
    and  $$|\trr|=O(\exp(-2\underline r)),$$ a simple computation (see Proposition 4.1 in \cite{neves2}) shows that  $$|f|_{C^2(S^2)}\leq C,$$
    where $C=C(C_1,C_2, C_3,C_4, r_1).$
    \end{proof}

\section{Existence and uniqueness of constant mean curvature foliations}\label{U}

In this section we show existence and uniqueness of foliations by stable spheres with constant mean curvature. In \cite{neves} this was accomplished via perturbing coordinate spheres so that they become constant mean curvature spheres. Unfortunately, this  cannot be immediately applied to our new setting because the quadratic terms do not seem to have the necessary decay.  

Nonetheless, the  {\em apriori} estimates derived in Lemma \ref{h}, Theorem \ref{2ff}, and Theorem \ref{sil3}, will allow us to continuously deform a constant mean curvature sphere in Anti-de Sitter-Schwarzschild space into a constant mean curvature sphere in our asymptotically hyperbolic metric. This method was used by Metzger in \cite{metzger} and we will adapt it to our setting.

We start with some estimates regarding the {\em normalized} Jacobi operator
\begin{equation*}\label{J} 
Lf =-\widehat \Delta f -|\Sigma|(4\pi)^{-1}\left(\lvert A\rvert^2+R(\nu,\nu) \right)f,
\end{equation*}
where the Laplacian is computed with respect to the normalized metric $\hat g$ defined in Theorem \ref{intrinsic}. The volume form with respect to this metric will be denoted by $d\hat \mu$.

\begin{prop}\label{invert}
Let $\phi$ be a solution to $L\phi=\alpha$, where $\alpha$ is a constant. There is $r_0=r_0(C_1,C_2, C_3,C_4, r_1)$ and $C=C(C_1,C_2, C_3,C_4, r_1)$ so that if 
$\underline r\geq r_0$, then
$$\fint_{\Sigma}(\phi-\bar \phi)^2d\hat \mu\leq C\exp(\underline r)\bar\phi\left(-\alpha-2\bar \phi+\bar \phi C\exp(-2\underline r)+\bar\phi\fint_{\Sigma}\frac{3\tr h}{2\sinh\hat r}d\bar \mu \right),$$
where $\bar \phi$ denotes the average (computed with respect to $\hat g$) of $\phi$.

Moreover, the operator $L$ is invertible and positive definite when restricted to functions with average zero.
\end{prop}
\begin{proof}
From Gauss equation, Lemma \ref{ads}, and Theorem \ref{sil3}, we have that
\begin{align*}
	|A|^2+Rc(\nu,\nu) & = 2K-R+3Rc(\nu,\nu)+3\n^2\\
	&= 2K-\frac{3\tr h}{2\sinh^ 3r}+O(\exp(-4\underline r))\\
	&= 2K-\frac{3\tr h}{2\sinh^ 3\hat r}+O(\exp(-4\underline r))
\end{align*}
and so
$$|\Sigma|(4\pi)^{-1}\left(\lvert A\rvert^2+R(\nu,\nu) \right)=2\widehat K-\frac{3\tr h}{2\sinh\hat r}+O(\exp(-2\underline r)).$$
Set $u=\phi-\bar \phi$. Then
$$Lu=\left(2\widehat K-\frac{3\tr h}{2\sinh\hat r}+O(\exp(-2\underline r))\right)\bar \phi+\alpha,$$
and thus, from integration by parts,
\begin{multline*}
\fint_{\Sigma}|\widehat \nabla u|^2d\hat\mu+\fint_{\Sigma}\left(\frac{3\tr h}{2\sinh\hat r}-2\widehat K+O(\exp(-2\underline r))\right)u^2d\hat\mu\\
=\bar \phi \fint_{\Sigma} \,u\left(2\widehat K-\frac{3\tr h}{2\sinh\hat r}+O(\exp(-2\underline r))\right) d\hat \mu.\\
\end{multline*}

Using the well know fact that that the lowest eigenvalue of $S^2$ is bounded below by $2\inf \widehat K$, we obtain that
\begin{multline*}
	\fint_{\Sigma}\left(\frac{3\tr h}{2\sinh\hat r} +2\inf\widehat K-2\widehat K+O(\exp(-2\underline r))\right)u^2d\hat\mu\\
	\leq \bar \phi \fint_{\Sigma} \,u\left(2\widehat K-\frac{3\tr h}{2\sinh\hat r}+O(\exp(-2\underline r))\right) d\hat \mu.\\
\end{multline*}
From Gauss equation, Lemma \ref{ads},  Theorem \ref{2ff},  Theorem \ref{sil3},  and the fact that $\tr h>0$, we know that
$$\widehat K-\inf \widehat K\leq  \frac{\tr h}{2\sinh\hat r}+C\exp(-2\underline r),$$
where $C=C(C_1,C_2, C_3,C_4, r_1)$.
Furthermore, integrating the equation satisfied by $u$, we obtain that
\begin{multline*}
 \fint_{\Sigma} \,u\left(2\widehat K-\frac{3\tr h}{2\sinh\hat r}+O(\exp(-2\underline r))\right) d\hat \mu\\
 =-\alpha-\bar \phi  \fint_{\Sigma} \left(2\widehat K-\frac{3\tr h}{2\sinh\hat r}+O(\exp(-2\underline r))\right) d\hat \mu\\
 =-\alpha-2\bar \phi+\bar\phi\fint_{\Sigma} \left(\frac{3\tr h}{2\sinh\hat r}+O(\exp(-2\underline r))\right) d\hat \mu.\\
 \end{multline*}
For this reason,
\begin{multline*}
	\fint_{\Sigma}\left(\frac{\tr h}{2\sinh\hat r} -C\exp(-2\underline r)\right)u^2d\hat\mu\\
	\leq \bar\phi\left(-\alpha-2\bar \phi+\bar \phi C\exp(-2\underline r)+\bar\phi\fint_{\Sigma}\frac{3\tr h}{2\sinh\hat r}d\hat \mu \right)
\end{multline*}
and so, if we choose $r_0$ sufficiently large, we can find $C=C(C_1,C_2, C_3,C_4,r_1)$ for which
$$\fint_{\Sigma}u^2d\hat\mu
	\leq C\exp(\underline r)\bar\phi\left(-\alpha-2\bar \phi+\bar \phi C\exp(-2\underline r)+\bar\phi\fint_{\Sigma}\frac{3\tr h}{2\sinh\hat r}d\hat \mu \right).
$$

What we have done so far also  shows that for  every function $f$ with average zero  we have,  provided we choose $r_0$ sufficiently large
$$\fint_{\Sigma} f Lf\,d\hat \mu\geq \fint_{\Sigma}\left(\frac{\tr h}{2\sinh\hat r} -C\exp(-2\underline r)\right)f^2d\hat\mu>0,$$ 
which shows the positive definiteness of $L$.

If $\phi$ is a solution of $L\phi=0$, then the first assertion of this proposition implies that

$$\fint_{\Sigma}(\phi-\bar\phi)^2d\hat\mu
	\leq C\exp(\underline r)\left(-2\bar \phi^2+\bar \phi^2 C\exp(-\underline r) \right),
$$
where $C=C(C_1,C_2, C_3,C_4, r_1)$. Thus, if we choose $r_0=r_0(C_1,C_2, C_3,C_4, r_1)$ sufficiently large and assume $\underline r\geq r_0$, then $\phi=\bar \phi$ and so $\phi=0$. Hence, $L$ is injective and consequently invertible.
\end{proof}

 We can now prove the main theorem of this paper. It is enough to show
 \begin{thm} Let $(M,g)$ be an asymptotically hyperbolic manifold satisfying hypothesis (H). Outside a compact set, $M$   admits a foliation by stable spheres with constant mean curvature. The foliation is unique among those with the property that, for some constant $C_4$, each leaf has
  $$\overline r-\underline r\leq C_4.$$
 \end{thm}
 \begin{proof}
 We first prove existence. Fix a coordinate system on $M\setminus K$ such that  condition \eqref{centered} is satisfied. The metric $g$ can be written as
 $$g=dr^2+\sinh^2 r g_0+h/(3\sinh r) + Q\quad\mbox{for} \quad{ r\geq r_1},$$
 where $Q$ satisfies the conditions of Definition \ref{defi}, and so we can consider  the family of asymptotically hyperbolic metrics $g_t$ given by
 $$g^t=dr^2+\sinh^2 r g_0+(th+(1-t)g_0)/(3\sinh r)+tQ+(1-t)P \quad\mbox{for} \quad{ r\geq r_1},$$
 where $P$ is such that $g^0$ is an Anti-de Sitter-Schwarzschild metric with positive mass.
 Note that condition \eqref{centered} is satisfied by all  metrics $g^t$ and that we can choose constants $C_1,C_2,$ and $C_3$ such that  hypothesis (H) is satisfied by  $g^t$ for all $0\leq t\leq 1$.
 
 Choose $r_0$ such that Theorem \ref{2ff}, Theorem \ref{sil3}, and Proposition \ref{invert} hold. We say that a constant mean curvature stable sphere $\Sigma$ satisfies hypothesis (A) or (B) if
 $$(A)\qquad \underline r\geq r_0\quad\mbox{and}\quad \overline r-\underline r\leq C_4,$$
 or
 $$(B)\qquad \underline r\geq 3r_0\quad\mbox{and}\quad \overline r-\underline r\leq C_4/3$$
 respectively.
 Note that due to Theorem \ref{sil3}, there is $D_1=D_1(C_1,C_2, C_3,C_4, r_1)$ such that $\underline r\geq \hat r-D_1$. Moreover, we see from Lemma \ref{h} that for all $\delta$ we can chose $\varepsilon_0=\varepsilon_0(C_1,C_2, C_3,C_4, r_1)$ such that
$$H\leq 2+\varepsilon_0\implies \hat r\geq \delta^{-1}.$$
As a result, there is $\varepsilon_0$ such that a sphere $\Sigma$ that satisfies hypothesis (A) and has $H\leq 2+\varepsilon_0$, also satisfies hypothesis (B).

Take $2<l<2+\varepsilon_0$ and set $$\mathcal S(l)=\{t\in [0,1]\,|\,\Sigma_t \mbox{ satisfies (A) and }H=l\mbox{ with respect to }g_t\}.$$
 This set is nonempty because contains $t=0$. If $t_0$ is in $\mathcal S(l)$,  we know from Proposition \ref{invert} that the linearization of the mean curvature is invertible  at $\Sigma_{t_0}$ with respect to the metric $g^{t_0}$. Thus, the inverse function theorem implies, for every $t$ sufficiently close to $t_0$, the existence of a  sphere $\Sigma_t$ with $H=l$, where the mean curvature is computed with respect to $g^t$. Another consequence of   Proposition \ref{invert} is that $\Sigma_{t_0}$ is strictly stable and this implies that, for all $t$ sufficiently close to  $t_0$, the spheres $\Sigma_t$ are also strictly stable. Because  $\Sigma_{t_0}$ satisfies hypothesis (B) with respect to $g^{t_0}$ we have that, for all $t$ sufficiently close to $t_0$,
 $$\overline r-\underline r\leq C_4/2\quad\mbox{and}\quad\underline r\geq 2r_0.$$
 This implies that $\mathcal S(l)$ is open. 
 
 Next we argue that $\mathcal S(l)$ is closed. Let $(t_i)_{i\in\N}$ be  a sequence in $\mathcal S$ converging to $\bar t$. From Theorem \ref{sil3}, each surface $\Sigma_i$ can be described as
 $$\Sigma_i=\{(\hat r_i+f_i(\theta),\theta)\,|\, \theta\in S^2\}\quad\mbox{with}\quad |f_i|_{C^2(S^2)}\leq C,$$
 where the sequence $(\hat r_i)$ must be bounded because the mean curvature is fixed. Each $f_i$  solves a quasilinear elliptic equation on $S^2$ and so standard theory implies a uniform $C^{2,\alpha}$--bound on each $f_i$. Therefore, after passing to a subsequence, $f_i$ converges to a function $f$ in $C^{2,\alpha}$ that satisfies the constant mean curvature equation with respect to $g^{\bar t}$ and such that its graph is a stable constant mean curvature sphere satisfying hypothesis (A). Thus, $\mathcal S(l)$ is closed and hence equal to $[0,1]$.
 
 Uniqueness is proven similarly. Suppose that the metric $g^{1}$ admits two distinct stable constant mean curvature spheres $\Sigma^1$ and $\Sigma^2$ satisfying hypothesis (A) and with $H(\Sigma_1)=H(\Sigma_2)$. Repeating the same arguments, we obtain the existence, for each $g^t$, of two families $\Sigma^1_t$ and $\Sigma^2_t$ of stable constant mean curvature spheres with equal mean curvature which can never coincide because, for all $0\leq t\leq 1$, the Jacobi operator is invertible (Proposition \ref{invert}). Therefore,  the Anti-de Sitter-Schwarzschild metric $g^0$ has two distinct stable constant mean curvature spheres satisfying hypothesis (A) with the same value for the mean curvature. This contradicts the uniqueness proven in \cite{neves}.  
 
 For each $2<l<2+\varepsilon_0$, denote by $\Sigma^l$ the unique stable sphere satisfying hypothesis (A)  with $H=l$ (with respect to the metric $g$). The fact that uniqueness holds combined with the fact that the Jacobi operator is invertible implies that $\Sigma_l$ constitutes a smooth family with respect to the parameter $l$. To check that it is indeed a foliation, we need to make sure that they never intersect.  We will show that is true for all $l$ sufficiently close to $2$.
     
     Denote by $V_l$ the deformation vector created by the family $(\Sigma^l)_{2<l<2+\varepsilon_0}$ and by $\phi$ its normal projection, i.e., $\phi=\langle V_l,\nu\rangle$. We want to show that $\phi$ does not change sign for all $l$ sufficiently close to 2. Gauss-Bonnet Theorem, Gauss equations, Lemma \ref{ads}, Theorem \ref{2ff}, and Theorem \ref{sil3}, imply that
     \begin{align*}
     		H^2(\Sigma^l) & =4+\frac{16\pi}{|\Sigma^l |}-2\fint_{\Sigma^l}\frac{\tr h}{\sinh^3 r} d\mu+O(\exp(-4\underline r))\\
						& =4+\frac{16\pi}{|\Sigma^l |}-2\left(\frac{4\pi}{|\Sigma^l|}\right)^{3/2}\fint_{S^2}\tr h d\mu_0+O(\exp(-4\underline r)).
     \end{align*}
Moreover, one can also check that
\begin{equation*}
	2H\frac{dH}{dl}=-\frac{d |\Sigma^l|}{dl}\frac{16\pi}{|\Sigma^l|^2}+3\frac{d |\Sigma^l|}{dl}\frac{(4\pi)^{3/2}}{|\Sigma^l|^{5/2}}\fint_{S^2}\tr h d\mu_0+O(\exp(-4\underline r))
\end{equation*}
and
$$\frac{d |\Sigma^l|}{dl}=H\int_{\Sigma^l} \phi d\mu=H|\Sigma^l|\bar\phi.$$
Therefore
\begin{equation}\label{barro}
	\frac{|\Sigma^l|}{4\pi}\frac{dH}{dl}=-2\bar\phi+\bar \phi\fint_{S^2}\frac{3\tr h}{2\sinh \hat r}d\mu_0+\bar\phi O(\exp(-2\underline r)).
\end{equation}
	On the other hand, we know that
	$$L\phi=\frac{|\Sigma^l|}{4\pi}\frac{dH}{dl},$$
	where the operator $L$ was defined in the beginning of this section. Thus, Proposition \ref{invert} implies that
	$$\fint_{\Sigma}(\phi-\bar \phi)^2d\hat \mu\leq C\exp(-\underline r)\bar\phi^2,$$
	where $C=C(C_1,C_2, C_3,C_4, r_1).$ Because $\phi$ solves  a linear elliptic equation, standard estimates imply that, for all  $\underline r$ sufficiently large,
	$$\sup |\phi-\bar \phi|\leq |\bar\phi |/2.$$
	Therefore, for all $l$ sufficiently close to $2$, $\phi$ does not change sign and this implies the theorem.
\end{proof}

\newpage

\bibliographystyle{amsbook}

\vspace{20mm}

\end{document}